\documentclass[10pt,conference]{IEEEtran}
\IEEEoverridecommandlockouts
\usepackage{needspace, amsthm}
\makeatletter
\def\ps@headings{%
\def\@oddhead{\mbox{}\scriptsize\rightmark \hfil \thepage}%
\def\@evenhead{\scriptsize\thepage \hfil \leftmark\mbox{}}%
\def\@oddfoot{}%
\def\@evenfoot{}}
\makeatother
\usepackage{amsfonts, amsmath}
\usepackage{url}
\usepackage{subfigure}
\usepackage{wrapfig}
\newlength{\smpagewidth}
\newlength{\smpageheight}

\newcommand{\setleftmargin}[1]{
	\addtolength{\textwidth}{\oddsidemargin}
	\addtolength{\textwidth}{1in}
	\addtolength{\textwidth}{-#1}
	\setlength{\oddsidemargin}{-1in}
	\addtolength{\oddsidemargin}{#1}
	\setlength{\evensidemargin}{\oddsidemargin}
}
\newcommand{\setrightmargin}[1]{
	\setlength{\textwidth}{\smpagewidth}
	\addtolength{\textwidth}{-\oddsidemargin}
	\addtolength{\textwidth}{-1in}
	\addtolength{\textwidth}{-#1}
}
\newcommand{\settopmargin}[1]{
	\addtolength{\textheight}{\topmargin}
	\addtolength{\textheight}{1in}
	\addtolength{\textheight}{\headheight}
	\addtolength{\textheight}{\headsep}
	\addtolength{\textheight}{-#1}
	\setlength{\topmargin}{-1in}
	\addtolength{\topmargin}{-\headheight}
	\addtolength{\topmargin}{-\headsep}
	\addtolength{\topmargin}{#1}
}
\newcommand{\setbottommargin}[1]{
	\setlength{\textheight}{\smpageheight}
	\addtolength{\textheight}{-\topmargin}
	\addtolength{\textheight}{-1in}
	\addtolength{\textheight}{-\footskip}
	\addtolength{\textheight}{-#1}
}

\usepackage{amsfonts}
\usepackage{amssymb}
\usepackage{acronym}
\usepackage{color}

\usepackage{graphicx}
\usepackage{color}

\newcommand{\field}[1]{\mathbb{#1}}

\def\l{\lambda}

\def\a{\alpha}

\def\rn{\mathbb{R}^n}

\def\re{\mathbb{R}}

\newtheorem{theorem}{Theorem}

\newtheorem{assumption}{Assumption}
\newtheorem{algorithm}{Algorithm}

\newtheorem{lemma}{Lemma}

\renewenvironment{proof}[1]{\needspace{1\baselineskip}\noindent {\bf Proof #1:} } {\hfill $\blacksquare$ \medskip}{}

\setleftmargin{.75in}
\setrightmargin{.75in}
\settopmargin{.80in}
\setbottommargin{.80in}

\begin{document}

%\title{A Distributed Newton Method for Network Optimization\thanks{This research was partially
%supported by the National Science Foundation under Career grant
%DMI-0545910, the DARPA ITMANET program, ARL MAST, DARPA/DSO
%SToMP, and  ONR MURI N000140810747.}}
%
%
%\author{Ali Jadbabaie, Asuman Ozdaglar, and Michael Zargham\thanks{Ali Jadbabaie and Michael Zargham are with the Department of Electrical
%and Systems Engineering, University of Pennsylvania.  Asuman Ozdaglar is with the Department of Electrical Engineering and Computer
%Science, Massachusetts Institute of Technology.}}

\title{\vspace{.1in}
\huge {A Distributed Line Search for  Network Optimization\thanks{This research is supported by Army Research Lab MAST
Collaborative Technology Alliance, AFOSR complex networks program, ARO P-57920-NS, NSF CAREER CCF-0952867, and NSF CCF-1017454, ONR MURI N000140810747 and NSF-ECS-0347285.}}}

%\author{Michael Zargham$^{\dagger}$, Alejandro Ribeiro$^{\dagger}$, Asuman Ozdaglar$^\ddagger$, Ali Jadbabaie$^{\dagger}$ \thanks{$^{\dagger}$Michael Zargham, Alejandro Ribeiro and Ali Jadbabaie are with the Department of Electrical and Systems Engineering, University of Pennsylvania.}
%\thanks{$^\ddagger$ Asuman Ozdaglar is with the Department of Electrical Engineering and Computer
%Science, Massachusetts Institute of Technology. }}

\author{Michael Zargham$^{\dagger}$, Alejandro Ribeiro$^{\dagger}$, Ali Jadbabaie$^{\dagger}$ \thanks{$^{\dagger}$Michael Zargham, Alejandro Ribeiro and Ali Jadbabaie are with the Department of Electrical and Systems Engineering, University of Pennsylvania.}
}

\maketitle

\thispagestyle{empty}
\begin{abstract}

Dual descent methods are used to solve network optimization problems because descent directions can be computed in a distributed manner using information available either locally or at neighboring nodes. However, choosing a stepsize in the descent direction remains a challenge because its computation requires global information. This work presents an algorithm based on a local version of the Armijo rule that allows for the computation of a stepsize using only local and neighborhood information. We show that when our distributed line search algorithm is applied with a descent direction computed according to the Accelerated Dual Descent method \cite{acc11}, key properties of standard backtracking line search using the Armijo rule are recovered. We use simulations to demonstrate that our algorithm is a practical substitute for its centralized counterpart.
\end{abstract}

\section{Introduction}

Conventional approaches to distributed network optimization are based on iterative descent in either the primal or dual domain. The reason for this is that for many types of network optimization problems there exist descent directions that can be computed in a distributed fashion. Subgradient descent algorithms, for example, implement iterations through distributed updates based on local information exchanges with neighboring nodes; see e.g., \cite{mung, kelly, low, Srikant}. However, practical applicability of the resulting algorithms is limited by exceedingly slow convergence rates typical of gradient descent algorithms. Furthermore, since traditional line search methods require global information, fixed stepsizes are used, exacerbating the already slow convergence rate, \cite{averagepaper, CISS-rate}.

Faster distributed descent algorithms have been recently developed by constructing approximations to the Newton direction using iterative local information exchanges, \cite{lowdiag,cdc09,acc11}. These results build on earlier work in \cite{BeG83} and \cite{cgNewt} which present Newton-type algorithms for network flow problems that, different from the more recent versions in \cite{cdc09} and \cite{acc11}, require access to all network variables. To achieve global convergence and recover quadratic rates of centralized Newton's algorithm \cite{cdc09} and \cite{acc11} use distributed backtracking line searches that use average consensus to verify global exit conditions. Since each backtracking line search step requires running a consensus iteration with consequently asymptotic convergence \cite{fagnani, spielman}, the exit conditions of the backtracking line search can only be achieved up to some error. Besides introducing inaccuracies, computing stepsizes with a consensus iteration is not a suitable solution because the consensus iteration itself is slow. Thus, the quadratic convergence rate of the algorithms in \cite{cdc09} and \cite{acc11} is to some extent hampered by the linear convergence rate of the line search. This paper presents a distributed line search algorithm based on local information so that each node in the network can solve its own backtracking line search using only locally available information.

Work on line search methods for descent algorithms can be found in \cite{NMLS,zhang,cgLS}. The focus in \cite{NMLS} and \cite{zhang} is on nonmonotone line searches which improve convergent rates for Newton and Newton-like descent algorithms. The objective in \cite{cgLS} is to avoid local optimal solutions in nonconvex problems. While these works provide insights for developing line searches they do not tackle the problem of dependence on information that is distributed through nodes of a graph.

To simplify discussion we restrict attention to the network flow problem. Network connectivity is modeled as a directed graph and the goal of the network is to support a single information flow specified by incoming rates at an arbitrary number of sources and outgoing rates at an arbitrary number of sinks. Each edge of the network is associated with a concave function that determines the cost of traversing that edge as a function of flow units transmitted across the link. Our objective is to find the optimal flows over all links. Optimal flows can be found by solving a concave optimization problem with linear equality constraints (Section II).  Evaluating a line search algorithm requires us to choose a descent direction. We choose to work with the family of Accelerated Dual Descent (ADD) methods introduced in \cite{acc11}. Algorithms in this family are parameterized by the information dependence between nodes. The $N$th member of the family, shorthanded as ADD-N, relies on information exchanges with nodes not more than $N$ hops away. Similarly, we propose a group of line searches that can be implemented through information exchanges with nodes in this $N$ hop neighborhood.

Our work is based on the Armijo rule which is the workhorse condition used in backtracking line searches, \cite[Section 7.5]{LaNLP}. We construct a local version of the Armijo rule at each node by taking only the terms computable at that node, using information from no more than $N$ hops away(Section III).  Thus the line search always has the same information requirements as the descent direction computed via the ADD-N algorithm.  Our proofs(Section IV) leverage the information dependence properties of the algorithm to show that key properties of the backtracking line search are preserved: (i) We guarantee the selection of unit stepsize within a neighborhood of the optimal value (Section IV-A). (ii) Away from this neighborhood, we guarantee a strict decrease in the optimization objective (Section IV-B). These properties make our algorithm a practical distributed alternative to standard backtracking line search techniques. Simulations further demonstrate that our line search is functionally equivalent to its centralized counterpart (Section V).

\section{Network Optimization}\label{mincost-Newton}

Consider a network represented by a directed graph ${\cal G}=({\cal N},{\cal E})$ with node set ${\cal N}=\{1,\ldots,n\}$, and edge set ${\cal E} = \{1,\ldots,E\}$.  The $i$th component of vector $x$ is denoted as $x^{i}$. The notation $x\ge 0$ means that all components $x^i\ge 0$. The network is deployed to support a single information flow specified by incoming rates $b^{i}>0$ at source nodes and outgoing rates $b^{i}<0$ at sink nodes. Rate requirements are collected in a vector $b$, which to ensure problem feasibility has to satisfy $\sum_{i=1}^{n}b^{i}=1$. Our goal is to determine a flow vector $x=[x^e]_{e\in {\cal E}}$, with $x^e$ denoting the amount of flow on edge $e=(i,j)$.

Flow conservation implies that it must be $Ax=b$, with $A$ the $n\times E$ node-edge incidence matrix defined as
\[ [A]_{ij} = \left\{
\begin{array}{ll}
 1 & \hbox{if edge $j$ leaves node $i$} , \\
-1 & \hbox{if edge $j$ enters node $i$},  \\
 0 & \hbox{otherwise,}
\end{array}\right.\]
where $[A]_{ij}$ denotes the element in the $i$th row and $j$th column of the matrix $A$. We define the reward as the negative of scalar cost function $\phi_e(x^{e})$ denoting the cost of $x^e$ units of flow traversing edge $e$. We assume that the cost functions $\phi_e$ are strictly convex and twice continuously differentiable. The maximum reward network optimization problem is then defined as
\begin{equation}
	\begin{array}{ll}\hbox{maximize }& -f(x)=\sum_{e=1}^E -\phi_e(x^e)\\ \hbox{subject to: }& Ax=b.\label{optnet}\end{array}
\end{equation}
Our goal is to investigate a distributed line search technique for use with Accelerated Dual Descent (ADD) methods for solving the optimization problem in (\ref{optnet}). We begin by discussing the Lagrange dual problem of the formulation in (\ref{optnet}) in Section \ref{subsection:dualform}) and reviewing the ADD method in Section \ref{subsection:ADD}.

\subsection{Dual Formulation} \label{subsection:dualform}
Dual descent algorithms solve (\ref{optnet}) by descending on the Lagrange dual function $q(\lambda)$. To construct the dual function consider the Lagrangian ${\cal L} (x,\l) = -\sum_{e=1}^E \phi_e(x^e) +\l'(Ax-b)$ and define
\begin{eqnarray} \label{eqn_separable_dual}
	q(\l)  \!\!&=&\!\! \sup_{x\in \re^E} {\cal L} (x,\l) \nonumber \\
                &=&\sup_{x\in \re^E} \left(-\sum_{e=1}^E \phi_e(x^e) +\l'Ax\right) - \l'b \nonumber\\
		       &=&\!\! \sum_{e=1}^E \sup_{x^e \in \re} \Big((\l'A)^e x^e-\phi_e(x^e)\Big) - \l'b,
\end{eqnarray}
where in the last equality we wrote $\l'Ax = \sum_{e=1}^{E}(\l'A)^e x^e$ and exchanged the order of the sum and supremum operators.

It can be seen from (\ref{eqn_separable_dual}) that the evaluation of the dual function $q(\l)$ decomposes into the $E$ one-dimensional optimization problems that appear in the sum. We assume that each of these problems has an optimal solution, which is unique because of the strict convexity of the functions $\phi_e$. Denote this unique solution as $x^e(\l)$ and use the first order optimality conditions for these problems in order to write
\begin{equation}
	x^e(\l) = (\phi_e')^{-1} (\l^i-\l^j),\label{primal-sol}
\end{equation}
where $i\in{\cal N}$ and $j\in{\cal N}$ respectively denote the source and destination nodes of edge $e=(i,j)$. As per (\ref{primal-sol}) the evaluation of $x^e(\l)$ for each node $e$ is based on local information
about the edge cost function $\phi^e$ and the dual variables of the incident nodes $i$ and $j$.

The dual problem of (\ref{optnet}) is defined as $\min_{\l\in \re^n} q(\l)$. The dual function is convex, because all dual functions of minimization problems are, and differentiable, because the $\phi_e$ functions are strictly convex. Therefore, the dual problem can be solved using any descent algorithm of the form
\begin{equation}
	\l_{k+1} = \l_k + \a_k d_k\qquad \hbox{for all }k\ge 0, \label{sgit}
\end{equation}
where the descent direction $d_k$ satisfies $g_k'd_k<0$ for all times $k$ with $g_k = g(\l_{k})=\nabla q(\l_{k})$ denoting the gradient of the dual function $q(\l)$ at $\l=\l_k$.  An important observation here is that we can compute the elements of $g_k$ as
 \begin{equation}\label{eqn_dual_update_distributed}
	g_k^{i} = \sum_{e =(i,j)} x^e(\l_{k}) -  \sum_{e =(j,i)} x^e(\l_{k})  - b_{i}.
\end{equation}
with the vector $x(\l_k)$ having components $x^e(\l_k)$ as determined by (\ref{primal-sol}) with $\l=\l_k$, \cite[Section 6.4]{nlp}.  An important fact that follows from \eqref{eqn_dual_update_distributed} is that the $i$th element $g_k^{i}$ of the gradient $g_k$ can be computed using information that is either locally available $x^{(i,j)}$ or available at neighbors $x^{(j,i)}$.  Thus, the simplest distributed dual descent algorithm, known as subgradient descent takes $d_k=-g_k$. Subgradient descent suffers from slow convergence so we work with an approximate Newton direction.

\subsection{Accelerated Dual Descent} \label{subsection:ADD}
The Accelerated Dual Descent (ADD) method is a parameterized family of dual descent algorithms developed in \cite{acc11}. An algorithm in the ADD family is called ADD-N and each node uses information from $N$-hop neighbors to compute its portion of an approximate Newton direction. Two nodes are $N$-hop neighbors if the shortest undirected path between those nodes is less than or equal to $N$.

The exact Newton direction $d_k$ is defined as the solution of the linear equation $ H_k d_k = -g_k$ where $H_k = H(\l_{k})=\nabla^2 q(\l_{k})$ denotes the Hessian of the dual function.  We approximate $d_k$ using the ADD-N direction defined as
\begin{equation}
d_k^{(N)} = -\bar H_k^{(N)} g_k \label{ADDd}
\end{equation}
where the approximate Hessian inverse, $\bar H_k^{(N)}$ is defined
\begin{equation}
\bar H_k^{(N)} =\sum_{r=0}^N D_k^{-\frac{1}{2}}\!\left(D_k^{-\frac{1}{2}} B_k D_k^{-\frac{1}{2}}\right)^r\!D_k^{-\frac{1}{2}}\!
\label{ADD_h}
\end{equation}
using a Hessian splitting: $H_k= D_k-B_k$ where $D_k$ is the diagonal matrix $[D_k]_{ii} = [H_k]_{ii}$.  The resulting accelerated dual descent algorithm
\begin{equation}
	\l_{k+1} = \l_k + \a_k d^{(N)}_k\qquad \hbox{for all }k\ge 0, \label{update}
\end{equation}
can be computed using information from $N$-hop neighbors because the dependence structure of $g_k$ shown in equation (\ref{eqn_dual_update_distributed}) causes the Hessian to have a local structure as well: $[H_k]_{ij} \not = 0$ if and only if $(i,j)\in \mathcal{E}$.  since $H_k$ has the sparsity pattern of the network, $B_k$ and thus $D_k^{-\frac{1}{2}} B_k D_k^{-\frac{1}{2}}$ must also have the sparsity pattern of the graph.  Each term $D_k^{-\frac{1}{2}}\!\left(D_k^{-\frac{1}{2}} B_k D_k^{-\frac{1}{2}}\right)^r\!D_k^{-\frac{1}{2}}\!$ is a matrix which is non-zero only for $r$-hop neighbors so the sum is non-zero only for $N$-hop neighbors.

Analysis of the ADD-N algorithm fundamentally depends on a network connectivity coefficient $\bar \rho$, which is defined in \cite{acc11} as the bound
\begin{equation}
 \rho \left(B_k D_k^{-1}\right) \le \bar\rho \in (0,1) \label{rhobar}
 \end{equation}
 where $\rho(\cdot)$ denotes the second largest eigenvalue modulus.  When $\bar\rho$ is small, information in the network spreads efficiently and $d_k^{(N)}$ is a more exact approximation of $d_k$.  See \cite{acc11} for details.

%LEFTOVER (save temporarily)
%We observe that since $H_k$ has the sparsity pattern of the network, $B_k$ and thus $D_k^{-\frac{1}{2}} B_k D_k^{-\frac{1}{2}}$
%must also have the sparsity pattern of the graph.  Each term $D_k^{-\frac{1}{2}}\!\left(D_k^{-\frac{1}{2}} B_k D_k^{-\frac{1}{2}}\right)^r\!D_k^{-\frac{1}{2}}\!$ is a matrix which is non-zero only for $r$-hop neighbors so the sum is non-zero only for $N$-hop neighbors.
%
%The resulting accelerated dual descent algorithm
%%
%\begin{equation}
%	\l_{k+1} = \l_k + \a_k d^{(N)}_k\qquad \hbox{for all }k\ge 0, \label{update}
%\end{equation}
%%
%can be computed using information from $N$-hop neighbors only.
%
%
%  Unfortunately, the inverse of the Hessian is in general a dense matrix. The ADD algorithm constructs an approximate inverse Hessian denoted $\bar H_k^{(N)}$ where $N$ is a parameter defining information dependence.  In particular $\left[\bar H_k^{(N)}\right]_{ij} \not = 0$ if and only if nodes $i$ and $j$ are $N$-hop neighbors.
\section{Distributed Backtracking Line Search}

Algorithms ADD-$N$ for different $N$ differ in their information dependence. Our goal is to develop a family of distributed backtracking line searches parameterized by the same $N$ and having the same information dependence. The idea is that the $N^{th}$ member of the family of line searches is used in conjunction with the $N^{th}$ member of the ADD family to determine the step and descent direction in (\ref{update}). As with the ADD-$N$ algorithm, implementing the distributed backtracking line search requires each node to get information from its $N$-hop neighbors.

Centralized backtracking line searches are typically intended as method to find a stepsize  $\alpha$ that satisfies Armijo's rule. This rule requires the stepsize $\alpha$ to satisfy the inequality
\begin{equation}q(\lambda+\alpha d) \le q(\lambda) + \sigma \alpha d'g, \label{armijo}\end{equation}
for given descent direction $d$ and search parameter $\sigma\in (0,1/2)$. The backtracking line search algorithm is then defined as follows:
\begin{algorithm}\label{BLS}
Consider the objective function $q(\cdot)$ and given variable value $\lambda$ and corresponding descent direction $d$ and dual gradient $g$. The backtracking line search algorithm is:
\begin{itemize}
\item[\footnotesize\bf] Initialize $\alpha=1$
\item[\footnotesize\bf] {\bf while} $q(\lambda+\alpha d) > q(\lambda) + \sigma \alpha d'g$
\item[\footnotesize\bf] $\qquad\alpha = \alpha \beta$
\item[\footnotesize\bf] \bf{end}
\end{itemize}
The scalars $\beta\in (0,1)$ and $\sigma \in (0,1/2)$ are given parameters.
\end{algorithm}

This line search algorithm is commonly used with Newton's method because it guarantees a strict decrease in the objective and once in an error neighborhood it always selects $\alpha=1$ allowing for quadratic convergence, \cite[Section 9.5]{boydbook}.

In order to create a distributed version of the backtracking line search we need a local version of the Armijo Rule.  We start by decomposing the dual objective $q(\lambda) = \sum_{i=1}^n q_i(\lambda)$ where the local objectives takes the form \begin{equation}q_i(\lambda)= \sum_{e=(j,i)} \phi_e(x^e)-\lambda_i(a_i'x-b_i).\label{ldep}\end{equation} The vector $a_i'$ is the $i^{th}$ row of the incidence matrix $A$. Thus the local objective $q_i(\lambda)$ depends only on the flows adjacent to node $i$ and $\lambda^i$.

An $N$-parameterized local Armijo rule is therefore given by
\begin{equation}
q_i(\lambda+\alpha_i d)\le q_i(\lambda)+\sigma\alpha_i\sum_{j\in \mathcal{N}_i^{(N)}} {d^j g^j},
\label{armijoN}
\end{equation}
where $\mathcal{N}_j^{(N)}$ is the set of $N$-hop neighbors of node $j$.  The scalar $\sigma\in (0,1/2)$ is the same as in (\ref{armijo}), $g=\nabla q(\lambda)$ and $d$ is a descent direction.  Each node is able to compute a stepsize $\alpha_i$
satisfying (\ref{armijoN}) using $N$-hop information.  The stepsize used for the dual descent update (\ref{sgit}) is
\begin{equation}
\alpha = \min_{i\in \mathcal{N}} \alpha_i.
\end{equation}
Therefore, we define the distributed backtracking line search according to the following algorithm.
%\begin{algorithm}
%Given local objectives $q_i(\lambda)$, descent direction $d$ and dual gradient $g$, let $\alpha = 1$ and iterate $\alpha = \alpha \beta$ until
%\[q_i(\lambda+\alpha d)\le q_i(\lambda)+\sigma\alpha\sum_{j\in \mathcal{N}_i^{(N)}} {d^j g^j}
%\]
%is satisfied for all $i$. The scalars $\beta\in (0,1)$, $\sigma \in (0,1/2-\bar\rho^{N+1}/2)$ and $N\in \field{Z}^+$ are parameters.
%\label{DBLS}
%\end{algorithm}

\begin{algorithm}
Given local objectives $q_i(\cdot)$, descent direction $d$ and dual gradient $g$.
\begin{itemize}
\item[\footnotesize\bf] \bf{for} $i=1:n$
\item[\footnotesize\bf] $\qquad$Initialize $\alpha_i=1$
\item[\footnotesize\bf] $\qquad${\bf while} $q_i(\lambda+\alpha_i d)> q_i(\lambda)+\sigma\alpha_i\sum_{j\in \mathcal{N}_i^{(N)}} {d^j g^j}$
\item[\footnotesize\bf] $\qquad\qquad\alpha_i = \alpha_i \beta$
\item[\footnotesize\bf] $\qquad$\bf{end}
\item[\footnotesize\bf] \bf{end}
\item[\footnotesize\bf] $\alpha = \min_i \alpha_i$
\end{itemize}
The scalars $\beta\in (0,1)$, $\sigma \in (0,1/2-\bar\rho^{N+1}/2)$ and $N\in \field{Z}^+$ are parameters.
\label{DBLS}
\end{algorithm}

The distributed backtracking line search described in Algorithm \ref{DBLS} works by allowing each node to execute its own modified version of Algorithm \ref{BLS} using only information from $N$-hop neighbors. Minimum consensus of $\alpha_i$ requires at most diameter of $\mathcal{G}$ iterations. If each node shares its current $\alpha_i$ along with $g_k^i$ with its $N$-hop neighbors the maximum number of iterations drops to $\lceil \hbox{diam}(\mathcal{G})/N \rceil$.

The parameter $\sigma$ is restricted by the network connectivity coefficient $\bar\rho$ and the choice of $N$ because these are scalars which encode information availability.  Smaller $\bar\rho^{N+1}$ indicates more accessible information and thus allows for greater $\sigma$ and thus a more aggressive search.  As $\bar\rho^{N+1}$ approaches zero, we recover the condition $\sigma\in(0,1)$ from Algorithm \ref{BLS}.

\section{Analysis}

In this section we show that when implemented with the Accelerated Dual Descent update in (\ref{update}) the distributed backtracking line search defined in Algorithm \ref{DBLS} recovers the key properties of Algorithm \ref{BLS}: strict decrease of the dual objective and selection of $\alpha=1$ within an error neighborhood.

We proceed by outlining our assumptions.  The standard Lipshitz and strict convexity assumptions regarding the dual Hessian are defined here.
\begin{assumption} \label{standard}
The Hessian $H(\l)$ of the dual function $q(\l)$ satisfies the following conditions %(IS THE GRADIENT NECESSARY HERE? IT DOES NOT SEEM SO)
\begin{list}{}{\setlength{\itemsep  }{2pt} \setlength{\parsep    }{2pt}
                                           \setlength{\parskip  }{0in} \setlength{\topsep    }{2pt}
                                           \setlength{\partopsep}{0in} \setlength{\leftmargin}{10pt}
                                           \setlength{\labelsep }{10pt} \setlength{\labelwidth}{-0pt}}
\item[({\it Lipschitz dual Hessian})] There exists some constant $L>0$ such that
\[\|H(\l)-H(\bar{\l})\| \le L\|\l-\bar{\l}\| \, \forall \l,\bar{\l}\in \rn.\]
\item [({\it Strictly convex dual function})] There exists some constant $M>0$ such that $\|H(\l)^{-1}\|\le M \qquad \, \forall \l\in \rn.$
\end{list}\end{assumption}
In addition to assumptions about the dual Hessian we assume that key properties of the inverse Hessian carry forward to our approximation.
\begin{assumption}
The approximate inverse Hessian remains well conditioned,
\[ m \le \|\bar H^{(N)}\| \le M.\]
within the subspace $\mathbf{1}^\perp$.
\label{cond}
\end{assumption}
These assumptions make sense because $\bar H^{(N)}$ is a truncated sum whose limit as $N$ approaches infinity is $H^{-1}$, a matrix we already assume to be well conditioned on $\mathbf{1}^\perp$ even when solving this problem in the centralized case.  Furthermore the first term in the sum is $D^{-1}$ which is well conditioned by construction.

We begin our analysis by characterizing the stepsize $\alpha$ chosen by Algorithm \ref{DBLS} when the descent direction $d$ is chosen according the the ADD-N method.
\begin{lemma}
For any $\alpha_i$ satisfying the distributed Armijo rule in equation (\ref{armijoN})
with descent direction  $d= -\bar H^{(N)}g$ we have
\[q_i(\lambda+\alpha_i d)- q_i(\lambda)\le 0.\]
\label{neg}\vspace{-5mm}
\end{lemma}
\begin{proof}
Recall that $\bar H ^{(N)}$ is non-zero only for elements corresponding to $N$-hop neighbors by construction. Therefore, by defining the local gradient vector $\tilde g^{(i)}$ as a sparse vector with nonzero elements $[\tilde g^{(i)}]_j=g^j$ for $j\in \mathcal{N}_i^{(N)}$ we can write
\begin{equation}\sum_{j\in \mathcal{N}_i^{(N)}} {d^j g^j} = -\left(\tilde g^{(i)}\right)' \bar H ^{(N)}\tilde g^{(i)}\label{quad}\end{equation}
Because $\bar H^{(N)}$ is positive definite the right hand side of (\ref{quad}) is nonpositive from where it follows that $\sum_{j\in \mathcal{N}_i^{(N)}} {d^j g^j}\leq0 $. The desired result follows by noting that $\alpha_i$ and $\sigma$ are positive scalars.\end{proof}

Lemma \ref{neg} tells us that when using the distributed backtracking line search with the ADD-N algorithm, we achieve improvement in each element of the decomposed objective $q_i(\lambda)$.  From the quadratic form in equation (\ref{quad}) it also follows that if equation (\ref{armijoN}) is satisfied by a stepsize $\alpha_i$, then it is also satisfied by any $\alpha\le \alpha_i$ and in particular $\alpha= \min_i \alpha_i$ satisfies equation (\ref{armijoN}) for all $i$.

%\begin{lemma}
%The minimum stepsize $\alpha = \min_i \alpha_i$ satisfies
%\[q_i(\lambda_{k+1})\le q_i(\lambda_k)+\sigma \alpha\sum_{j\in \mathcal{N}_i^{(N)}} {d_k^j g_k^j}\]
%for all $i$.
%\end{lemma}
%
%\begin{proof}
%Using the fact that
%$\sum_{j\in \mathcal{N}_i^{(N)}} {d_k^j g_k^j}\le 0$ as per the quadratic form in (\ref{quad}), we observe that for each $i$,
%\[\alpha_i\sum_{j\in \mathcal{N}_i^{(N)}} {d_k^j g_k^j}\le \alpha\sum_{j\in \mathcal{N}_i^{(N)}} {d_k^j g_k^j},\]
%bringing us to the desired relation.
%\end{proof}
%This result indicates that if the individual $\alpha_i$ stepsizes satisfy (\ref{armijoN}) we can replace $\alpha_i$ with $\alpha = \min_i \alpha_i$ and condition (\ref{armijoN}) is still met for all $i$.

\subsection{Unit Stepsize Phase}

A fundamental property of the backtracking line search using Armijo's rule summarized in Algorithm \ref{BLS}is that it always selects $\alpha=1$ when iterates $\lambda$ are within a neighborhood of the optimal argument. This property is necessary to ensure quadratic convergence of Newton's method and is therefore a desirable property for the distributed line search summarized in Algorithm \ref{DBLS}. We prove here that this is true as stated in the following theorem.

\begin{theorem}
Consider the distributed line search in Algorithm \ref{DBLS} with parameter $N$, starting point $\lambda=\lambda_k$, and descent direction $d = d_k^{(N)} = -\bar H_k^{(N)} g_k$ computed by the ADD-$N$ algortihm [cf. \eqref{ADDd} and \eqref{ADD_h}. If the search parameter $\sigma$ is chosen such that
\[ \sigma\in \left(0, \frac{1-\bar\rho^{N+1}}{2}\right) \]
and the norm of the dual gradient satisfies
\[\|g_k\|\le\frac{3m}{LM^3}\left({1-\bar\rho^{N+1}}-2\sigma\right), \]
then Algorithm \ref{DBLS} selects stepsize $\alpha=1$.
\label{a1}
\end{theorem}
\begin{proof}
Recall the definition of the local gradient $\tilde g^{(i)}_k$ as the sparse vector with nonzero elements $[\tilde g^{(i)}_k]_j=g_k^j$ for $j\in \mathcal{N}_i^{(N)}$. Further define the local update vector $\tilde d_k^{(i)} := \bar H^{(N)}_k \tilde g^{(i)}_k$ whose sparsity pattern is the same as that of $\tilde g^{(i)}_k$. Due to this and to the fact that the local objective $q_i(\lambda)$ in \eqref{ldep} depends only on values in $\mathcal{N}_i^{(N)}$, we have
\begin{align}\label{eqn_theo_unit_step_size_pf_10}
    q_i(\lambda_k+\alpha d_k) = q_i(\lambda_k+\alpha\tilde d_k^{(i)}).
\end{align}
Applying the Lipschitz dual Hessian assumption to the local update vector $\tilde d_k^{(i)}$ we get
\begin{align}\label{eqn_theo_unit_step_size_pf_20}
    \Vert H(\lambda_k+\alpha \tilde d_k^{(i)})- H(\lambda_k)\Vert \le
        \alpha L \Vert \tilde d_k^{(i)}\Vert.
\end{align}
We further define a reduced Hessian $\nabla^2 q_i(\lambda) = \tilde H^{(i)}$ by setting to zero the rows and columns corresponding to nodes outside of the neighborhood $\mathcal{N}_i^{(N)}$, i.e.,
\begin{equation}
    \left[\tilde H^{(i)}\right]_{ij} :=
        \left\{ \begin{array}{ll} H_{ij} & i,j\in \mathcal{N}_i^{(N)}\\
                                  0 & else\end{array}\right. \label{RH}
\end{equation}
Since the elements of $H$ already satisfy $H_{ij}= 0$ for all ${i,j}\not\in \mathcal{E}$ the resulting $\tilde H^{(i)}$ has the structure of a principal submatrix of $H$ with the deleted rows left as zeros. Since the norm $\Vert H(\lambda_k+\alpha \tilde d_k^{(i)})- H(\lambda_k)\Vert$ in \eqref{eqn_theo_unit_step_size_pf_20} is the maximum eigenvalue modulus of the matrix $H(\lambda_k+\alpha \tilde d_k^{(i)})- H(\lambda_k)$, it is larger than the norm $\Vert \tilde H^{(i)}(\lambda_k+\alpha \tilde d_k^{(i)})- \tilde H^{(i)}(\lambda_k)\Vert$ because the latter is the maximum over a subset of the eigenvalues of the former. Combining this observation with \eqref{eqn_theo_unit_step_size_pf_20} yields
\begin{equation}\Vert \tilde H^{(i)}(\lambda_k+\alpha \tilde d_k^{(i)})- \tilde H^{(i)}(\lambda_k)\Vert\le \alpha L\Vert \tilde d_k^{(i)}\Vert.\label{LL}\end{equation}
Interpret now the update in \eqref{eqn_theo_unit_step_size_pf_10} as a function of $\tilde q_i(\alpha)$ defined as
\begin{equation}\label{eqn_theo_unit_step_size_pf_50}
    \tilde q_i(\alpha) := q_i(\lambda_k+\alpha\tilde d_k^{(i)}).
\end{equation}
Differentiating with respect to $\alpha$ and using the definition of the local gradient $\tilde g^{(i)}_k$ we get the derivative of $\tilde q_i(\alpha)$ as
\begin{equation}\label{eqn_theo_unit_step_size_pf_60}
    \tilde q_i'(\alpha)
        = \nabla q_i (\lambda_k+\alpha\tilde d_k^{(i)}) \tilde d_k^{(i)}
        = \tilde g^{(i)}(\lambda_k+\alpha\tilde d_k^{(i)}) \tilde d_k^{(i)}.
\end{equation}
Differentiating with respect to $\alpha$ a second time and using the definition of $\tilde H^{(i)}$ in \eqref{RH} yields
\begin{align}\label{eqn_theo_unit_step_size_pf_70}
    \tilde q_i''(\alpha)
       &=\ \tilde d_k^{(i)}\phantom{}' \nabla^2 q_i(\lambda_k+\alpha\tilde d_k^{(i)})\tilde d_k^{(i)}\nonumber\\
       &=\ \tilde d_k^{(i)}\phantom{}' \tilde H^{(i)}(\lambda_k+\alpha\tilde d_k^{(i)})\tilde d_k^{(i)}.
\end{align}
Return now to (\ref{LL}) and replace the matrix norm on the right hand side with left and right multiplication by the unit vector $\tilde d_k^{(i)}/\|\tilde d_k^{(i)}\|$. This yields
\begin{align}\label{eqn_theo_unit_step_size_pf_80}
    \tilde d_k^{(i)}\phantom{}'
        \left[\tilde H^{(i)}(\lambda_k+\alpha \tilde d_k^{(i)})\!- \tilde H^{(i)}(\lambda_k)\right]
            \tilde d_k^{(i)}
    \!\le \alpha L\Vert \tilde d_k^{(i)}\Vert^3.
\end{align}
Comparing the expressions for the derivatives $\tilde q_i''(\alpha)$ in \eqref{eqn_theo_unit_step_size_pf_70} with the left hand side of \eqref{eqn_theo_unit_step_size_pf_80} we can simplify the latter to
\[\tilde q_i''(\alpha)-\tilde q_i''(0)\le \alpha L\Vert \tilde d_k^{(i)}\Vert^3.\]
Integrating the above expression with respect to $\alpha$ results in
\[\tilde q_i'(\alpha)-\tilde q_i'(0)\le \frac{\alpha^2}{2} L\Vert \tilde d_k^{(i)}\Vert^3+\alpha\tilde q_i''(0),\]
which upon a second integration with respect to $\alpha$ yields
\[\tilde q_i(\alpha)-\tilde q_i(0)\le \frac{\alpha^3}{6} L\Vert \tilde d_k^{(i)}\Vert^3+\frac{\alpha^2}{2}\tilde q_i''(0)+\alpha\tilde q_i'(0).\]
Since we are interested in unit stepsize substitute $\alpha=1$ and the definitions of the derivatives $\tilde q_i'(0)$ and $\tilde q_i''(0)$ given in \eqref{eqn_theo_unit_step_size_pf_60} and \eqref{eqn_theo_unit_step_size_pf_70} to get
\begin{equation*}
    \tilde q_i(1)-\tilde q_i(0)\le
         \frac{L}{6} \Vert \tilde d_k^{(i)}\Vert^3 \!
       + \frac{1}{2} \tilde d_k^{(i)}\phantom{}'\tilde H^{(i)}(\lambda_k)\tilde d_k^{(i)}
       + \tilde g_k^{(i)}\phantom{}' \tilde d_k^{(i)}.
\end{equation*}
Since according to (\ref{RH}) the reduced Hessian $\tilde H^{(i)}$ has the structure of a principal submatrix of the Hessian $H$ and $H\succeq 0$ it follows that $0\preceq\tilde H^{(i)}\preceq H$ and that as a consequence \[\tilde d_k^{(i)}\phantom{}' \tilde H^{(i)}(\lambda_k)\tilde d_k^{(i)} \le \tilde d_k^{(i)}\phantom{}' H_k \tilde d_k^{(i)}.\]
Incorporating this latter relation and the definition of the local update $\tilde d_k^{(i)} = \bar H^{(N)}_k \tilde g^{(i)}_k$ in the previous equation we obtain
\begin{align}\label{eqn_theo_unit_step_size_pf_140}
   \tilde q_i(1)-\tilde q_i(0)\le & \
        \frac{L}{6}\Vert \bar H^{(N)}_k \tilde g^{(i)}_k\Vert^3 \\&\hspace{-4mm}\nonumber
      + \frac{1}{2}\left( \bar H^{(N)}_k \tilde g^{(i)}_k\right)' H_k\bar H^{(N)}_k \tilde g^{(i)}_k
      - \tilde g_k^{(i)}\phantom{}' \bar H^{(N)}_k \tilde g^{(i)}_k.
\end{align}
Consider now the last term in the right hand side and recall the sparsity pattern of the local gradient $\tilde g_k^{(i)}$ to write
\begin{align}\label{eqn_theo_unit_step_size_pf_150}
    -\tilde g_k^{(i)} \bar H^{(N)}_k \tilde g^{(i)}_k = \sum_{j\in \mathcal{N}_i^{(N)}} g_k^j d_k^j,
\end{align}
and further split the right hand side of \eqref{eqn_theo_unit_step_size_pf_150} to generate suitable structure
\begin{align}\label{eqn_theo_unit_step_size_pf_160}
    \sum_{j\in \mathcal{N}_i^{(N)}} g_k^j d_k^j
        = \sum_{j\in \mathcal{N}_i^{(N)}} \sigma{g_k^j d_k^j} + (1-\sigma){g_k^j d_k^j}.
\end{align}
Substitute now \eqref{eqn_theo_unit_step_size_pf_160} into \eqref{eqn_theo_unit_step_size_pf_150} and the result into \eqref{eqn_theo_unit_step_size_pf_140} to write
\begin{align*}
    \tilde q_i(1)-\tilde q_i(0)\le \
         \frac{L}{6}&\Vert \bar H^{(N)}_k \tilde g^{(i)}_k\Vert^3
       + \frac{1}{2}\tilde d_k^{(i)}\phantom{}' H\tilde d_k^{(i)}\\&
       + \sigma \sum_{j\in \mathcal{N}_i^{(N)}} {g_k^j d_k^j}
       + (1-\sigma)\sum_{j\in \mathcal{N}_i^{(N)}}{g_k^j d_k^j}.
\end{align*}
Using the expression for the quadratic form in (\ref{eqn_theo_unit_step_size_pf_150}) to substitute the last term in the previous equation yields
\begin{align}\label{terms}
    \tilde q_i(1)-\tilde q_i(0)\le \
          \frac{L}{6}&\Vert\bar H_k ^{(N)}\tilde g_k^{(i)}\Vert^3
         +\frac{1}{2}\tilde d_k^{(i)}\phantom{}' H\tilde d_k^{(i)}\\&
         +\sigma \sum_{j\in \mathcal{N}_i^{(N)}} {g_k^j d_k^j}
         -(1-\sigma)\tilde g^{(i)}_k\phantom{}' \bar H_k ^{(N)}\tilde g_k^{(i)}\nonumber
\end{align}
Further note that from the definition of $\tilde d^{(i)}$ it follows that
\[\tilde d_k^{(i)}\phantom{}' H_k\tilde d_k^{(i)} = \tilde g_k^{(i)}\phantom{}'\bar H_k^{(N)} H_k\bar H_k^{(N)}\tilde g_k^{(i)}.\]
The right hand side of this latter equality can be bounded using Cauchy-Schwarz's inequality and the submultiplicity of matrix norms as
\[\tilde g_k^{(i)}\phantom{}'\bar H_k^{(N)} H\bar H_k^{(N)}\tilde g_k^{(i)}\le \|\tilde g_k^{(i)}\|\, \| \bar H_k^{(N)}\|\, \|H_k \bar H_k^{(N)}\|\, \|\tilde g_k^{(i)}\|.\]
The norm $H_k \bar H_k^{(N)}$ can be further bounded using the result $\|H_k\bar H^{(N)}_k\|\le \bar \rho^{N+1}+1$ from \cite{acc11}. The norm $\|\bar H^{(N)}_k\|$ can be bounded as $\|\bar H^{(N)}_k\|\le M$ according to Assumption \ref{cond}. These two observations substituted in the last displayed equation yield
\begin{equation}\tilde d_k^{(i)}\phantom{}' H_k\tilde d_k^{(i)} \le M(\rho^{N+1}+1)\|\tilde g_k^{(i)}\|^2. \label{term2}\end{equation}
Applying the bound $\|\bar H^{(N)}_k\|\le M$ from Assumption 2 to the norm $\Vert\bar H_k ^{(N)}\tilde g_k^{(i)}\Vert^3$ we get $\Vert\bar H_k ^{(N)}\tilde g_k^{(i)}\Vert^3 \le M^3 \|\tilde g_k^{(i)}\|^3$.  Since Assumption 2 also guarantees that $\|\bar H^{(N)}_k\|\ge m$, we have \[\frac{\tilde g_k^{(i)}\phantom{}'\bar H_k^{(N)} \tilde g_k^{(i)}}{\|\tilde g_k^{(i)}\|^2} \ge m.\]  Therefore, we can write
\begin{equation}
\Vert\bar H_k ^{(N)}\tilde g_k^{(i)}\Vert^3 \le \frac{M^3}{m}\|\tilde g_k^{(i)}\|\ \tilde g_k^{(i)}\phantom{}'\bar H_k^{(N)} \tilde g_k^{(i)}.
\label{term1}
\end{equation}
Substituting the relations (\ref{term2}) and (\ref{term1}) in relation (\ref{terms}) and factoring we get
\begin{align*}
   & \tilde q_i(1)-\tilde q_i(0)\le \
          \sigma\sum_{j\in \mathcal{N}_i^{(N)}} {g_k^j d_k^j}\\&\hspace{3mm}
        + \tilde g^{(i)}_k\phantom{}' \bar H_k ^{(N)}\tilde g_k^{(i)}
              \left[-(1-\sigma)+\frac{LM^3}{6m}\Vert \tilde g_k^{(i)}\Vert+ \frac{\bar\rho^{N+1}}{2} + 1\right].
\end{align*}
Use $\Vert \tilde g_k^{(i)}\Vert\le \|g_k\|\le {6m}/({LM^3})(({1-\bar\rho^{N+1}})/{2}-\sigma)$ to write
\begin{align*}
  & \tilde q_i(1)-\tilde q_i(0)\le
        \sigma\sum_{j\in \mathcal{N}_i^{(N)}} {g_k^j d_k^j} \\&
      + M \Vert \tilde g_k^{(i)}\Vert^2
         \left[-(1-\sigma)
            +\left(\frac{1-\bar\rho^{N+1}}{2}-\sigma\right)
            +\frac{\bar\rho^{N+1}+1}{2} \right].
\end{align*}
Algebraic simplification of the bracketed portion yields
\begin{equation}\left[-(1-\sigma)+\left(\frac{1-\bar\rho^{N+1}}{2}-\sigma\right)+\frac{\bar\rho^{N+1}+1}{2} \right]=0\label{cancel}.\end{equation}
Thus we have \[\tilde q_i(1)-\tilde q_i(0)\le \sigma\sum_{j\in \mathcal{N}_i^{(N)}} {g_k^j d_k^j}.\]
Substituting the definition of $\tilde q_i(\lambda)$ in \eqref{eqn_theo_unit_step_size_pf_50} into this equation we arrive at
\[q_i\left(\lambda_{k}+d_k^{(i)}\right)\le q_i(\lambda_k)+\sigma\sum_{j\in \mathcal{N}_i^{(N)}} {d_k^j g_k^j}\]
which means that the exit condition \eqref{armijoN} in Algorithm \ref{DBLS} is met with $\alpha=1$.
\end{proof}
Theorem \ref{a1} guarantees that for an appropriately chosen line search parameter $\sigma$ the local backtracking line search will always choose a step size of $\alpha=1$ once the norm of the dual gradient becomes small.  Furthermore, the condition on the line search parameter tells us that $\bar\rho$ and our choice of $N$ fully capture the impact of distributing the line search. The distributed Armijo rule requires $\left(1-\bar\rho^{N+1}-2\sigma\right)>0$ while the standard Armijo rule requires $(1-2\sigma)>0$.  It is clear that in the limit $N\rightarrow \infty$ these conditions become the same with a rate controlled by $\bar\rho$.

\subsection{Strict Decrease Phase}
A second fundamental property of the backtracking line search with the Armijo rule is that there is a strict decrease in the objective when iterates are outside of an arbitrary noninfinitesimal neighborhood of the optimal solution. This property is necessary to ensure global convergence of Newton's algorithm as it ensures the quadratic convergence phase is eventually reached. Our goal here is to prove that this strict decrease can be also achieved using the distributed backtracking line search specified by Algorithm \ref{DBLS}.

Traditional analysis of the centralized backtracking line search of Algorithm \ref{BLS} leverages a lower bound on the stepsize $\alpha$ to prove strict decrease. We take the same approach here and begin by finding a global lower bound on the stepsize $\hat\alpha\le\alpha_i$ that holds for all nodes $i$. We do this in the following lemma.
\begin{lemma}
Consider the distributed line search in Algorithm \ref{DBLS} with parameter $N$, starting point $\lambda=\lambda_k$, and descent direction $d = d_k^{(N)} = -\bar H_k^{(N)} g_k$ computed by the ADD-$N$ algortihm [cf. \eqref{ADDd} and \eqref{ADD_h}. The stepsize
\[\hat \alpha = {2(1-\sigma)}\frac{m^2}{M^2}\]
satisfies the local Armijo rule in \eqref{armijoN}, i.e.,
\[q_i(\lambda_{k+1})\le q_i(\lambda_k)+\sigma  \hat{\alpha}\sum_{j\in n_i^{(N)}} {d_k^j g_k^j}\]
for all network nodes $i$ and all $k$.
\label{ahat}
\end{lemma}
\begin{proof}
From the mean value theorem centered at $\lambda_k$ we can write the dual function's value as
\[q_i(\lambda_k+\alpha \tilde d_k^{(i)}) = q_i(\lambda_k)+\alpha\tilde g_k^{(i)}\phantom{}' \tilde d_k^{(i)}+\frac{\alpha^2}{2} \tilde d_k^{(i)}\phantom{}' \tilde H^{(i)} (z) \tilde d_k^{(i)}\]
where the vector $z= \lambda_k + t\alpha\tilde d_k^{(i)}$ for some $t\in (0,1)$; see e.g.,\cite[Section 9.1]{boydbook}.  We use the relation $0\preceq\tilde H^{(i)}\preceq H$ and the bound $\|H^{-1}\|>m$ from Assumption \ref{cond} to transform this equality into the bound
\[q_i(\lambda_k+\alpha \tilde d_k^{(i)}) \le q_i(\lambda_k)+\alpha\tilde g_k^{(i)}\phantom{}' \tilde d_k^{(i)}+\frac{\alpha^2}{2 m} \Vert\tilde d_k^{(i)}\Vert^2.\]
Introduce now a splitting of the term $\alpha\tilde g_k^{(i)} \tilde d_k^{(i)}$ to generate convenient structure
\begin{align*}
    q_i(\lambda_k+\alpha \tilde d_k^{(i)}) \le\ &
          q_i(\lambda_k)\\& \hspace{-8mm}
        +\alpha\sigma\tilde g_k^{(i)}\phantom{}' \tilde d_k^{(i)}
        + \alpha(1-\sigma)\tilde g_k^{(i)}\phantom{}' \tilde d_k^{(i)}
        +\frac{\alpha^2}{2 m}\Vert\tilde d_k^{(i)}\Vert^2.
\end{align*}
Further apply the definition of  the local update vector $\tilde d_k^{(i)} := \bar H^{(N)}_k \tilde g^{(i)}_k$ and use the well-conditioning of the approximate inverse Hessian $\bar H_k^{(N)}$ as per Assumption \ref{cond} to claim that $m\le \|\bar H_k^{(N)}\|\le M$ and obtain
\begin{align*}
   q_i(\lambda_k+\alpha \tilde d_k^{(i)}) & \le \
          q_i(\lambda_k) \\& \hspace{-12mm}
        +\alpha\sigma\tilde g_k^{(i)}\phantom{}' \tilde d_k^{(i)}
        - \alpha(1-\sigma)m\|\tilde g_k^{(i)}\|^2
        + \frac{\alpha^2M^2}{2 m}  \Vert\tilde g_k^{(i)}\Vert^2.
\end{align*}
Factoring common terms in this latter equation yields
\begin{align*}
    q_i(\lambda_k+\alpha \tilde d_k^{(i)}) & \le \
        q_i(\lambda_k)\\&\hspace{-8mm}
      + \alpha\sigma\tilde g_k^{(i)}\phantom{}' \tilde d_k^{(i)}
      +\alpha m\|\tilde g_k^{(i)}\|^2 \left[-(1-\sigma)+\frac{\alpha M^2}{(2m^2)}\right].
\end{align*}
Substitute $\hat \alpha$ for $\alpha$ in this inequality. Observe that by doing so we have $[-(1-\sigma)+{\hat\alpha M^2}/({2m^2})]=0$ implying that the second term vanishes from this expression. Therefore
\[q_i(\lambda_k+\hat\alpha \tilde d_k^{(i)}) \le q_i(\lambda_k)+\hat\alpha\sigma\tilde g_k^{(i)}\phantom{}' \tilde d_k^{(i)}.\]
From the definitions of the local gradient $\tilde g^{(i)}_k$ as the sparse vector with nonzero elements $[\tilde g^{(i)}_k]_j=g_k^j$ for $j\in n_i^{(N)}$ and the local update vector $\tilde d_k^{(i)} := \bar H^{(N)}_k \tilde g^{(i)}_k$ the desired result follows. \end{proof}
We proceed with the second main result using Lemma \ref{ahat}, in the same manner that strict decrease is proven for the Newton method with the standard backtracking line search in \cite{boydbook}[Section 9.5].
\begin{theorem}
Consider the distributed line search in Algorithm \ref{DBLS} with parameter $N$, starting point $\lambda=\lambda_k$, and descent direction $d = d_k^{(N)} = -\bar H_k^{(N)} g_k$ computed by the ADD-$N$ algortihm [cf. \eqref{ADDd} and \eqref{ADD_h}]. If the norm of the dual gradient is bounded away from zero as $\|g_k\|\ge \eta$, the function  value at $\l_{k+1} = \l_k + \a_k d^{(N)}_k$ satisfies
\[q(\lambda_{k+1})-q(\lambda_k) \le -\beta \hat \alpha \sigma m N \eta^2\]
\label{strict}
I.e., the dual function decreases by at least $\alpha \sigma m N \eta^2$
\end{theorem}
\begin{proof}
According to Lemma \ref{ahat} we have
\[q_i(\lambda_{k+1})-q_i(\lambda_k) \le \hat \alpha \sigma \tilde g^{(i)}_k\phantom{}' \tilde d^{(i)}\] because $\hat \alpha$ is a lower bound on $\alpha_i$. Therefore, Algorithm \ref{DBLS} exits with $\alpha\in (\beta \hat \alpha, \hat\alpha)$ and any $\alpha\le \hat\alpha$ satisfies the exit condition in (\ref{armijoN}) therefore
\[q_i(\lambda_{k+1})-q_i(\lambda_k) \le \beta \hat \alpha \sigma \tilde g^{(i)}_k\phantom{}' \tilde d^{(i)}.\]
Applying Assumption \ref{cond} with the definition of $\tilde d^{(i)}$ we get
\[q_i(\lambda_{k+1})-q_i(\lambda_k) \le -\beta\hat \alpha \sigma m\|\tilde g^{(i)}_k\|^2.\]
Summing over all $i$,
\[q(\lambda_{k+1})-q(\lambda_k) \le -\beta\hat \alpha \sigma m \sum_{i=1}^n \|\tilde g^{(i)}_k\|^2.\]
Using the definition of the 2-norm we can write $\sum_{i=1}^n \|\tilde g^{(i)}\|^2 = \sum_{i=1} \sum_{j\in n_i^{(N)}} (g_k^j)^2$. Counting the appearance of each $(g_k^j)^2$ term in this  sum we have that $\sum_{i=1} \sum_{j\in n_i^{(N)}} (g_k^j)^2 = \sum_{i=1}|n_i^{(N)}| (g_k^i)^2$. Notice however that since the network is connected it must be $|n_i^{(N)}|\ge N$, from where it follows $\sum_{i=1} \sum_{j\in n_i^{(N)}} (g_k^j)^2 \leq N \sum_{i=1}^n (g^i_k)^2$. Substituting this expression into the above equation yields
\[q(\lambda_{k+1})-q(\lambda_k) \le -\beta\hat \alpha \sigma m N\sum_{i=1}^n (g^i_k)^2\]
Observe now that $\sum_{i=1}^n (g^i)^2=  \|g_k\|^2$ and substitute the lower bound $\eta\le \|g_k\|$ to obtain the desired relation.
\end{proof}

\begin{figure}[t]
\includegraphics[width=\columnwidth]{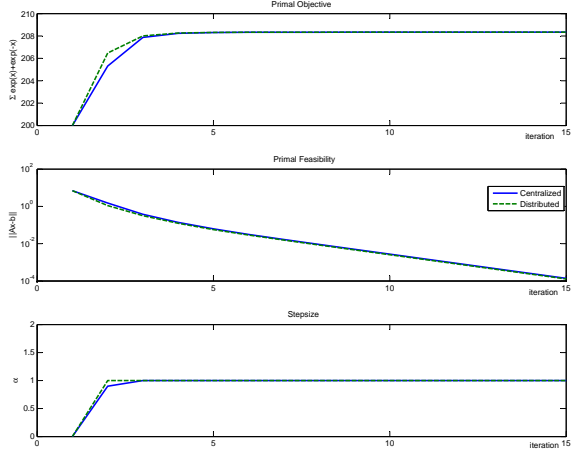}
\centering
\caption{\label{ex} The distributed line search results in solution trajectories nearly equivalent to those of the centralized line search. Top: the Primal Objective follows a similar trajectory in both cases.  Middle: Primal Feasibility is achieved asymptotically. Bottom: unit stepsize is achieved in roughly the same number of steps.}
\end{figure}

Theorem \ref{strict} guarantees global convergence into any error neighborhood $\|g_k\|\le \eta$ around the optimal value because the dual objective is strictly decreasing by, at least, the noninfinitesimal quantity $\beta\hat \alpha \sigma m N \eta^2$ while we remain outside of this neighborhood. In particular, we are guaranteed to reach a point inside the neighborhood $\|g_k\| \le \eta = 3m/(LM^3)\left({1-\bar\rho^{N+1}}-2\sigma\right)$ at which point  Theorem \ref{a1} will be true and the ADD-N algorithm with the local line search becomes simply
\[ \lambda_{k+1} = \lambda_k - \bar H^{(N)}_k g_k.\]
This iteration is shown to have quadratic convergence properties in \cite{acc11}.

\begin{figure}[t]
\includegraphics[width=\columnwidth]{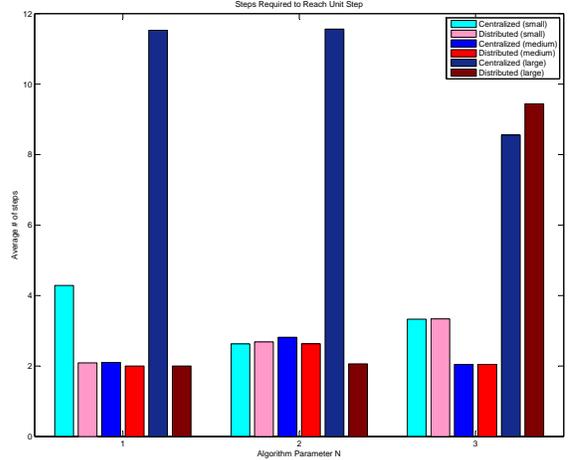}
\centering
\caption{
\label{data}The distributed line search reaches unit stepsize in 2 to 3 iterations.  Fifty simulations were done for each algorithm with N=1, N=2 and N=3 and for Networks with 25 nodes and 100 edges (small), 50 nodes and 200 edges (medium) and 100 nodes and 400 edges (large).}
\end{figure}

\section{Numerical results}
Numerical experiments demonstrate that the distributed version of the backtracking line search is functionally equivalent to the centralized backtracking line search when the descent direction is chosen by the ADD method.  The simulations use networks generated by selecting edges are added uniformly at random but are restricted to connected networks.   The primal objective function is given by $\phi^e(x) = e^{c x^e}+e^{-c x^e}$ where $c$ captures the notion of edge capacity.  For simplicity we let $c=1$ for all edges.

Figure \ref{ex} shows an example of a network optimization problem with 25 nodes and 100 edges being solved using ADD-1 with the centralized and distributed backtracking line searches.  The top plot shows that the trajectory of primal objective is not significantly affected by the choice line search.  The middle plot shows that primal feasibility is approached asymptotically at the same rate for both algorithms. The bottom plot shows that a unit stepsize is achieved in roughly the same number of steps.

In Figure \ref{data} we look closer at the number of steps required to reach a unit stepsize.  We compare the distributed backtracking line search to its centralized counterpart on networks with 25 nodes and 100 edges, 50 nodes and 200 edges and 100 nodes and 400 edges.  For each network optimization problem generated we implemented distributed optimization using ADD-1, ADD-2, and ADD-3.  Most trials required only 2 or 3 iterations to reach $\alpha=1$ for both the centralized and distributed line searches.  The variation came from the few trials which required significantly more iterations.  As might be expected, increasing $N$ causes the distributed and centralized algorithms to behave closer to each other. When we increase the size of the network most trials still only require 2 to 3 iterations to reach $\alpha=1$ but for the cases which take more than 2 iterations we jump from around 10 iterations in the 25 nodes networks to around 40 iterations in 100 node networks.

\section{Conclusion}
We presented an alternative version of the backtracking line search using a local version of the Armijo rule which allows the stepsize for the dual update in the single commodity network flow problem to be computed using only local information.  When this distributed backtracking line search technique is paired with the ADD method for selecting the dual descent direction we recover the key properties of the standard centralized backtracking line search: a strict decrease in the dual objective and unit stepsize in a region around the optimal.  We use simulations to demonstrate that the distributed backtracking line search is functionally equivalent to its centralized counterpart.

This work focuses on line searches when the ADD-N method is used to select the descent direction, however the proof method relies primarily on the sparsity structure of the inverse hessian approximation.  This implies that our line search method could be applied with other descent directions provided they have are themselves depend only on local information.

\vfill

\bibliographystyle{amsplain}
\bibliography{distributed}

\providecommand{\bysame}{\leavevmode\hbox to3em{\hrulefill}\thinspace}
\providecommand{\MR}{\relax\ifhmode\unskip\space\fi MR }
\providecommand{\MRhref}[2]{%
  \href{http://www.ams.org/mathscinet-getitem?mr=#1}{#2}
}
\providecommand{\href}[2]{#2}
\begin{thebibliography}{10}

\bibitem{lowdiag}
S.~Authuraliya and S.~H. Low, \emph{Optimization flow control with newton-like
  algorithm}, Telecommunications Systems \textbf{15} (2000), 345--358.

\bibitem{BeG83}
Bertsekas and Gafni, \emph{Projected newton methods and optimization of
  multi-commodity flow}, IEEE Transactions on Automatic Control \textbf{28}
  (1983), 1090--1096.

\bibitem{nlp}
D.P. Bertsekas, \emph{Nonlinear programming}, Athena Scientific, Cambridge,
  Massachusetts, 1999.

\bibitem{boydbook}
S.~Boyd and L.~Vandenberghe, \emph{Convex optimization}, Cambridge University
  Press, Cambridge, UK, 2004.

\bibitem{spielman}
M.~Cao, D.A. Spielman, and A.S. Morse, \emph{A lower bound on convergence of a
  distributed network consensus algorithm}, Proceedings of IEEE CDC, 2005.

\bibitem{fagnani}
R.~Carli, F.~Fagnani, A.~Speranzon, and S.~Zampieri, \emph{Communication
  constraints in coordinated consensus problems}, Proceedings of IEEE ACC,
  2006.

\bibitem{mung}
M.~Chiang, S.H. Low, A.R. Calderbank, and J.C. Doyle, \emph{Layering as
  optimization decomposition: A mathematical theory of network architectures},
  Proceedings of the IEEE \textbf{95} (2007), no.~1, 255--312.

\bibitem{cgLS}
W.~Hager and H.~Zhang, \emph{A new conjugate gradient method with guaranteed
  descent and an efficient line search}, SIAM journal of Optimization
  \textbf{16} (2005), 170--192.

\bibitem{cdc09}
A.~Jadbabaie, A.~Ozdaglar, and M.~Zargham, \emph{A distributed newton method
  for network optimization}, Proceedings of IEEE CDC, 2009.

\bibitem{kelly}
F.P. Kelly, A.K. Maulloo, and D.K. Tan, \emph{Rate control for communication
  networks: shadow prices, proportional fairness, and stability}, Journal of
  the Operational Research Society \textbf{49} (1998), 237--252.

\bibitem{cgNewt}
J.~G. Klincewicz, \emph{A newton method for convex separable network flow
  problems}, Bell Laboratories (1983).

\bibitem{low}
S.~Low and D.E. Lapsley, \emph{Optimization flow control, {I}: Basic algorithm
  and convergence}, IEEE/ACM Transactions on Networking \textbf{7} (1999),
  no.~6, 861--874.

\bibitem{LaNLP}
D.~G. Luenberger, \emph{Linear and nonlinear programming}, Klewer Academic
  Publishers, Boston, 2003.

\bibitem{averagepaper}
A.~Nedi\'c and A.~Ozdaglar, \emph{Approximate primal solutions and rate
  analysis for dual subgradient methods}, SIAM Journal on Optimization,
  forthcoming (2008).

\bibitem{CISS-rate}
\bysame, \emph{Subgradient methods in network resource allocation: Rate
  analysis}, Proc. of CISS, 2008.

\bibitem{NMLS}
G.~Di Pillo, \emph{On nonmonotone line search}, Journal of Optimization Theory
  and its Applications \textbf{112}, 315--330.

\bibitem{Srikant}
R.~Srikant, \emph{Mathematics of {I}nternet congestion control}, Birkhauser,
  2004.

\bibitem{acc11}
M.~Zargham, A.~Ribeiro, A.~Ozdaglar, and A.~Jadbabaie, \emph{Accelerated dual
  descent for network optimization}, Proceedings of IEEE ACC, 2011.

\bibitem{zhang}
H.~Zhang and W.~Hager, \emph{a nonmonotone line search technique and its
  application to unconstrained optimization}, SIAM journal of Optimization
  \textbf{14} (2004), 1043--1056.

\end{thebibliography}
\end{document}